\documentclass{article}
\usepackage{amsmath}
\usepackage{amsfonts,amssymb,mathrsfs}
\usepackage{theorem}
\usepackage{diagbox}
\usepackage{tikz}
\usepgflibrary{arrows}
\usetikzlibrary{matrix}
\usepackage{tikz-cd}
\usepackage{ textcomp }

\usepackage{graphicx}
\graphicspath{ {./images/} }

\usepackage{xcolor}
\usepackage[urlcolor  = blue]{hyperref}

\newcommand{\MYhref}[3][blue]{\href{#2}{\color{#1}{#3}}}

\input amssym.def

\numberwithin{equation}{section}

\numberwithin{figure}{section}

\makeatletter
\newcommand\qedsymbol{\hbox{$\Box$}}
\newcommand\qed{\relax\ifmmode\Box\else
  {\unskip\nobreak\hfil\penalty50\hskip1em\null\nobreak\hfil\qedsymbol
  \parfillskip=\z@\finalhyphendemerits=0\endgraf}\fi}

\newenvironment{proof-of}[1][{}]{\par\noindent \textbf{Proof of} {#1}. }{\qed}

\newenvironment{proof}[0]{\par\noindent \textbf{Proof}.}{\qed}

\newcommand{\W}{\mathsf{W}} 
\newcommand{\RP}{\mathsf{RP}} 

\newcommand{\PSL}{\mathrm{PSL}}

\newcommand{\bbZ}{{\mathbb Z}}

\renewcommand{\mod}{\,\mathrm{mod}\,}

\date{}
\newtheorem{thm}{Theorem}[section]
\newtheorem{defi}[thm]{Definition}

\newtheorem{lem}[thm]{Lemma}
\newtheorem{cor}[thm]{Corollary}

\newtheorem{prop}[thm]{Proposition}

\newtheorem{nota}[thm]{Notation}

\theorembodyfont{\rm}
\newtheorem{remark}[thm]{Remark}

\DeclareMathOperator{\rad}{rad}

\title{On the Replacement Property for $\PSL(2,p)$}

\author{David Cueto Noval, Aidan A. Lorenz and Baran Zadeoğlu}
\date{}

\begin{document}

\large

\maketitle

\begin{abstract}
The replacement property (or Steinitz Exchange Lemma) for vector spaces has a natural analog for finite groups and their generating sets. For the special case of the groups $\PSL(2,p)$, where $p$ is a prime larger than 5, first partial results 
concerning the replacement property were published by Benjamin Nachman \cite{Nachman}. Second partial results were published by Lam \cite{Lam}.
The main goal of this paper is to provide a complete answer for $\PSL(2,p)$. 

\end{abstract}

\section{Introduction}
\label{sec:intro}

There is an ongoing effort to create a theory for groups and their generating sequences, analogous to the theory of vector spaces and their respective bases; see \cite{ColThesis}, \cite{Lam}, \cite{Lucchini}, \cite{Nachman}. 
In detail, for a given group $G$, a sequence $s = (g_{1}, ..., g_{n}) \in G^n$ such that $G$ is generated by the $\{g_i\}_{i=1}^n$
is called a \emph{generating sequence of length $n$}; if no proper subsequence of $s$ generates $G$, then $s$
is called \emph{irredundant}. 
The largest possible length of an irredundant generating sequence of $G$ will be denoted by $m(G)$. 
The group $G$ is said to satisfy the \emph{replacement property} (abbr. $\RP$) if any $1 \neq g \in G$
can replace an element in all irredundant generating sequences of length $m(G)$ to yield
a new generating sequence of $G$. 
This property, an obvious analog of the Steinitz Exchange Lemma, does not generally hold for groups. Note that in the case of groups, the new generating sequence need not be irredundant. This definition is motivated further below.

This paper focuses on the groups $\PSL(2,p)$, where $p$ is a prime number $> 5$.  
Our main tool is an analysis of the maximal subgroups $\PSL(2,p)$ and their intersections. 
The following theorem summarizes our findings. 

\begin{thm}
\label{Theorem1}
The groups $\PSL(2,p)$ with $p \in \{ 7, 11, 19, 31\}$ satisfy $\RP$.
For all other primes $p>5$, $\PSL(2,p)$ satisfies  $\RP$ if and only if $p \equiv 3$ or $-3$ mod $8$ and $p \equiv 3$ or $-3$ mod $10$. In other words, we have the following:
\begin{center}
	\begin{tabular}{ l | c | r }
	\label{table}
		\backslashbox{$\mod8$}{$\mod10$} & $p \equiv \pm1$ & $p \equiv \pm3$ \\ \hline
    		$p \equiv \pm1$ & $\RP$ fails & $\RP$ fails \\ \hline
    		$p \equiv \pm3$ & $\RP$ fails & $\RP$ holds \\
		\end{tabular}
\end{center}
\end{thm}

Even though $\RP$ fails for the majority of primes, examples of failure are rare in these cases,
in the sense that most elements of $\PSL(2,p)$ can still replace an element in every irredundant generating sequence of length $m(G)$. An element that fails to do so will be called a \emph{witness to failure}. 

\begin{thm}
\label{Theorem2}
Witnesses to failure for $\PSL(2,p)$ must have order $2$ or $3$; if  $p \not\equiv \pm 1 \mod10$, 
they must have order $2$.
\end{thm}

The rare occurrences of witnesses to failure can be observed empirically via computer algebra systems such as GAP \cite{GAP}. 
In fact, checking whether a finite group satisfies $\RP$ can be done computationally. All the code used for this paper can be found on \MYhref{https://github.com/aidanlorenz/GeneratingSetsOfFiniteGroups}{the second author's GitHub page}.
For the majority of this paper, proof methodology is elementary and applicable for other classes of groups provided that $m(G)$ and the isomorphism classes of maximal subgroups are known. 
However, the latter half of the proof of the Theorem \ref{Theorem1} uses extensive knowledge about the subgroup lattice of the group. This can be reduced to character theoretic level using Mackey's theorem and other considerations of \cite{KomiTsar}.

\section{Notational conventions and definitions}

\begin{defi}
For a finite group $G$, r($G$) (resp. m($G$)) denotes the length of the shortest (resp. longest) irredundant generating sequence of $G$.
\end{defi}

\begin{nota}
$\Gamma_{n}(G)$ will denote the set of all irredundant generating sequences of $G$ of length $n$.
\end{nota}

Applying a more general theorem of Tarski \cite{Tarski} to groups, D. Collins was able to show that for all $n$ with $r(G) \leq n \leq m(G)$, $\Gamma_{n}(G) \neq \emptyset$.

\begin{defi}
A sequence of subgroups $(H_{1}, ..., H_{n})$ of a group is said to be in \textbf{general position} if, for all $j \in I = \{1, ..., n\}$, we have 
$$\bigcap_{i \neq j} H_{i} \not \subseteq H_{j}.$$ 
\end{defi}

We think of maximal subgroups in general position as group theoretic analogs of hyperplanes in general position: maximal subobjects which become strictly smaller upon intersection.

\begin{defi}
For $S = (M_{1}, ..., M_{n})$ a sequence of maximal subgroups of a group, we say a sequence $s = (g_{1}, ..., g_{n})$ \textbf{corresponds} to $S$ if $g_{i} \in M_{j}$ for all $j \neq i$, but $g_{i} \notin M_{i}$.
\end{defi}

Given an irredundant generating sequence $s = (g_{1}, ..., g_{n})$ of a finite group $G$, we can construct a corresponding sequence of maximal subgroups in the following fashion: let $$H_{i} := \langle g_{1}, ..., g_{i-1}, g_{i+1}, ..., g_{n} \rangle$$ Since $G$ is finite, each $H_{i}$ is contained in some maximal subgroup $M_{i}$. We thus associate to the irredundant generating sequence $s=(g_{1}, ..., g_{n})$, the sequence of maximal subgroups $S:=(M_{1}, ...., M_{n})$. It is straight forward to see that $S$ is in general position.

\begin{remark}
Though we can associate any irredundant generating sequence with a corresponding sequence of maximal subgroups in general position, we typically cannot do the converse.
One can observe that given a sequence of maximal subgroups in general position $S$, any sequence corresponding to $S$ will be irredundant. However these irredundant sequences are not necessarily generating sequences. 
The question remains open as to when we can make the converse association. 
\end{remark}

\begin{defi}
Given a sequence of maximal subgroups $S = (M_{1}, ..., M_{n})$, we call the \textbf{radical of S} (denoted $\rad(S)$) the intersection of all $M_{i}$: $$ \rad(S) := \bigcap_{1 \leq i \leq n} M_{i}.$$
\end{defi}

In each of the following two definitions, $G$ is a finite group, and $s = (g_{1}, ..., g_{n})$ is an irredundant generating sequence of $G$.

\begin{defi}
\label{RPfors}
$G$ satisfies the \textbf{replacement property relative to s} if for all $1 \neq g \in G$, there exists $i \in \{1, ..., n\}$ such that $s' := (g_{1}, ..., g_{i-1}, g, g_{i+1}, ..., g_{n})$ generates $G$.
\end{defi}

\begin{remark}
\label{equivdef}
An alternative definition to \ref{RPfors} replaces ``$g \neq 1$" with ``$g \notin \Phi(G)$" where $\Phi(G)$ denotes the Frattini subgroup of $G$, which is well known to consist of all non-generators of $G$. For our purposes in studying $\PSL(2,p)$, we have that $\Phi(\PSL(2,p)) = \{1\}$ for all $p$, so the definitions are equivalent.
\end{remark}

Thanks to an unpublished result from R.K. Dennis and D. Collins, we know that for a given finite group $G$ and for all $n$ with $r(G) \leq n < m(G)$, there exists some $s \in \Gamma_{n}(G)$ such that $G$ does not satisfy the replacement property relative to $s$. This allows us to define the replacement property in full generality:

\begin{defi}
\label{RP}
Let $m = m(G)$. We say $G$ satisfies the \textbf{replacement property} (abbreviated $\RP$) if $G$ satisfies the replacement property for all $s \in \Gamma_{m}(G)$.
\end{defi}

\begin{remark}
\label{quotientdef}
In the case when $\Phi(G) \neq \{1\}$, one could also use the definition that $G$ satisfies $\RP$ if and only if $G/\Phi(G)$ satisfies $\RP$ (in the sense of \ref{RP}). 
This is in accordance with Remark \ref{equivdef}.
\end{remark}

Notice, in Definition \ref{RPfors}, we do not require $s'$ to be irredundant. If we did, this would quickly lead to the result that for any group $G$ satisfying $\RP$ we would have that $r(G) = m(G)$, which too strictly limits the groups that might possibly enjoy this property. This can be demonstrated by the group $S_4$, an irredundant generating set containing transpositions and a cycle of length 4. 

Some examples of groups that satisfy $\RP$ include $M_{11}$ \cite{Nachman} and $S_{n}$ (for $n>6$ it follows from Theorem 2.1 in \cite{CamCar}). 

\begin{nota}
Given a finite group $G$, let $m = m(G)$. We call the subset $$\mathsf{W}(G) := \{g \in G \mid \exists s \in \Gamma_{m}(G) \text{ such that } g \text{ cannot replace any element of s}\}$$ the set of \textbf{witnesses to failure}.
\end{nota}

\section{An equivalent condition and applications to PSL(2,p)}

\begin{prop}[R.K. Dennis \& D. Collins (Unpublished)]
\label{iffCondition}
For a finite group $G$, let $m = m(G)$. Then $G$ satisfies $\RP$ if and only if for every sequence of maximal subgroups in general position of length $m$ corresponding to some irredundant generating sequence, $S$, we have that rad$(S) = \{1\}$.
\end{prop}

\begin{proof}
We prove the ``if" direction by contrapositive: assume $G$ does not satisfy $\RP$. Then there exists an irredundant generating sequence $s = (g_{1}, ..., g_{m})$ and an element $1 \neq g \in G$ such that for no $i \in \{1,..,n\}$ can $g_{i}$ be replaced by $g$ to yield a generating sequence. Now for each $i$, let $$ \langle g_{1}, ..., g_{i-1}, g, g_{i+1}, ..., g_{m} \rangle =: H_{i} \lneqq G$$ with proper containment because $s$ was irredundant and $g$ fails the replacement property relative to $s$. For each $H_{i}$, pick a maximal subgroup $M_{i}$ such that $H_{i} \leq M_{i}$. Clearly, the sequence $(M_{1}, ..., M_{m})$ corresponds to $s$ (thus, is in general position), and by construction, $g \in M_{i}$ for all $i$.

For the ``only if" direction, assume there exists a sequence of maximal subgroups in general position of length $m$, $S = (M_{1}, ..., M_{m})$ corresponding to some irredundant generating sequence of $G$, $s = (g_{1}, ..., g_{m})$, where rad$(S) \neq \{1\}$ so there exists a nontrivial element, $x \in \rad(S)$. Then, $$ \langle g_{1}, ..., g_{i-1}, x, g_{i+1}, ..., g_{m} \rangle \leq M_{i} \lneqq G$$ because $x \in M_{i}$ by assumption. Since $i$ was arbitrary, we know that $x$ fails the replacement property relative to $s$, and hence $G$ does not satisfy $\RP$.
\end{proof}
\vspace{2mm}
For a sequence $S$ of maximal subgroups in general position corresponding to an irredundant generating sequence we have that $\rad(S) \subseteq \W(G)$.
It is known that for any finite non-abelian simple group $G$, $r(G) = 2$. A further result by Jambor \cite{Jambor} follows:
\begin{thm}[Jambor \cite{Jambor}]
\label{Jambor}
$m(\PSL(2,p)) = 3$ for all primes except 7, 11, 19, and 31 in which cases we have $m(\PSL(2,p)) = 4$.
\end{thm}

Another important result that will be useful in our discussion is the classification of isomorphism types of maximal subgroups of $\PSL(2,p)$ from \cite{Dickson}:
\begin{thm}
\label{Dickson}
All maximal subgroups of $\PSL(2,p)$ are isomorphic to one of the following:
\begin{enumerate}
	\item $\bbZ_{p} \rtimes \bbZ_{\frac{p-1}{2}}$
	\item $D_{p-1}$
	\item $D_{p+1}$
	\item $A_{5}$ if and only if $p \equiv \pm 1 \mod 10$
	\item $S_{4}$ if and only if $p \equiv \pm 1 \mod 8$
	\item $A_{4}$ if and only if $p \equiv \pm 3 \mod 10$ \text{ and } $p \equiv \pm 3 \mod 8$
\end{enumerate}
where we use the convention that $|D_{n}| = n$.
\end{thm}

In fact, Dickson studied all subgroups of $\PSL(2,p)$. One can find a modern version of his work in \cite{King}.

For the isolated cases when $p \in \{7, 11, 19, 31\}$, B. Nachman showed that $\PSL(2,p)$ satisfies $\RP$ in \cite{Nachman}. In the same paper, Nachman showed  that for primes that are congruent to $+1 \mod 8$, $\PSL(2,p)$ does not satisfy $\RP$. Later, Ravi Fernando came up with the conjecture that is our Theorem \ref{Theorem1}.

Before we prove our main theorem, we state and prove some helpful lemmas.

\begin{lem}
\label{norms}
If $G$ is a simple group, $M_{1}$ and $M_{2}$ maximal subgroups of $G$, and there exists $N \neq \{1\}$ such that $N \trianglelefteq M_{1},M_{2}$, then $M_{1} = M_{2}$.
\end{lem}
\begin{proof}
Trivial.
\end{proof}

\begin{lem}
\label{ChainOfSubs}
For a triple of maximal subgroups $S=(M_{1}, M_{2}, M_{3})$ to be in general position it must be the case that:
\begin{enumerate}
	\item all pairwise intersections are nontrivial
	\item all of the maximal subgroups $M_{i}$ must have 2 distinct chains of nontrivial subgroups of length at least 3 (where we say the length of a chain is the number of non-trivial subgroups involved, including $M_{i}$ itself)
	\end{enumerate}
\end{lem}
\begin{proof}
The first statement is trivial. As for the second, let $(M_1, M_2, M_3)$ be a sequence of maximal subgroups in general position. The chains $M_2 > M_1 \cap M_2 > M_1 \cap M_2 \cap M_3$ and $M_2 > M_3 \cap M_2 > M_1 \cap M_2 \cap M_3$ are distinct chains of length 3 as $(M_1, M_2, M_3)$ are in general position.
\end{proof}

\begin{lem}
\label{Zps}
Let $G = \PSL(2,p)$. For $p>5$, there does not exist an element of order $p$ in $\W(G)$.
\end{lem}
\begin{proof}
Suppose there exists some $x \in G$ with order $p$. Then $x$ must lie in some maximal subgroup with order divisible by $p$. From Theorem \ref{Dickson}, we can see that $x$ must lie in a copy of $\bbZ_{p} \rtimes \bbZ_{\frac{p-1}{2}}$ . Hence, $\langle x \rangle \cong \bbZ_{p} \trianglelefteq \bbZ_{p} \rtimes \bbZ_{\frac{p-1}{2}}$. But by Lemma \ref{norms}, there is only one such maximal subgroup. Thus $x$ cannot be in the intersection of 3 maximal subgroups in general position. Therefore $x \notin \W(G)$.
\end{proof}
\vspace{2mm}
We can now return to our main theorem. Its proof is divided in several cases. In Proposition \ref{RPHolds}, we consider the case where, $p \equiv 3$ or $-3 \mod 8$ and $p \equiv 3$ or $-3 \mod 10$. In Proposition \ref{RPfail}, we consider the rest of the cases.

\begin{prop}
\label{RPHolds}
If $p \equiv \pm 3 \mod 8$ and $p \equiv \pm 3 \mod 10$, then $\PSL(2,p)$ satisfies $\RP$.
\end{prop}
\begin{proof}
According to Theorem \ref{Dickson}, the possible maximal subgroups in this case are 
\begin{enumerate}
\label{possibles}
	\item $\bbZ_{p} \rtimes \bbZ_{\frac{p-1}{2}}$
	\item $D_{p-1}$
	\item $D_{p+1}$
	\item $A_{4}$
\end{enumerate}
We will show that no triple of these maximal subgroups can be in general position while still having a nontrivial radical, thus proving the theorem by Proposition \ref{iffCondition}. 

Firstly, looking at the subgroup lattice, all chains of nontrivial subgroups of length 3 in $A_4$ contain the unique normal subgroup $\bbZ_{2} \times \bbZ_{2}$ as the middle term.
Therefore, by Lemma \ref{ChainOfSubs}, $A_{4}$ cannot appear in any triple of maximal subgroups in general position with nontrivial intersection.

Next, suppose we have a sequence of maximal subgroups in general position $S = (M_1, M_2, M_3)$ where each $M_i$ is isomorphic to $\bbZ_{p} \rtimes \bbZ_{\frac{p-1}{2}}$ or $D_{p-1}$ and suppose there is a nontrivial element in $\rad(S)$. By Lemma \ref{Zps}, we know that this element can't have order $p$. Assume further that there exists an element $x \in M_1 \cap M_2$ such that $\vert x \vert>2$. Then,
\begin{equation}
C_{M_1}(x), C_{M_2}(x) \leq C_{G}(x) \lneq G.
\end{equation}
In fact, we have $C_{M_1}(x) = C_{M_2}(x) = M_1 \cap M_2 = C_{G}(x)$. Now, if $\rad(S)$ or $M_1 \cap M_3$ contains an element of order larger than 2, then by the argument above we have that
\begin{equation}
M_1 \cap M_2 = M_1 \cap M_3 = \rad(S)
\end{equation}
which contradicts the general position assumption. Now assume $\rad(S)$ contains only elements of order 2 and the identity. If we have $M_1 \cong \bbZ_{p} \rtimes \bbZ_{\frac{p-1}{2}}$ then,
\begin{equation}
    \rad(S) \lneq M_1 \cap M_2 , M_1 \cap M_3 \lneq \bbZ_{p} \rtimes \bbZ_{\frac{p-1}{2}}.
\end{equation}
In particular, both $M_1 \cap M_2$ and $M_1 \cap M_3$ contains elements of order larger than 2. We again get the above contradiction. Now assume all maximal subgroups $M_i$ are isomorphic to $D_{p-1}$ and their pairwise intersections have no element of order larger than 2. Then, we must have that $M_i \cap M_j \cong \bbZ_{2} \times \bbZ_{2}$ for $i \neq j$ and $\rad(S) \cong \bbZ_{2}$. However, all isomorphic copies of $\bbZ_{2} \times \bbZ_{2}$ in a dihedral group must contain the group's center. In particular,
\begin{equation}
    Z(M_1) \subset M_1 \cap M_2 , M_1 \cap M_3
\end{equation}
Hence, $Z(M_1) = Z(M_2) = \rad(S)$ which implies $\PSL(2,p)$ has a center, contradiction. The same arguments also applies to the cases where each $M_i$ are isomorphic copies of $D_{p+1}$ or $\bbZ_{p} \rtimes \bbZ_{\frac{p-1}{2}}$.

Finally, we realize that $D_{p-1}$ and $D_{p+1}$ cannot appear in the same triple of maximal subgroups in general position with non-trivial radical since any intersection of $D_{p-1}$ and $D_{p+1}$ can have order at most 2, leaving the intersection with the third maximal subgroup to be trivial. 

We have therefore showed that no triple of maximal subgroups from the list of possibilities can be in general position and have nontrivial radical simultaneously. The proof is thereby complete.
\end{proof}

\begin{cor}
\label{orders}
Witnesses to failure in $\PSL(2,p)$ have order 2 or 3.
\end{cor}
\begin{proof}
Notice that in Proposition \ref{RPHolds}, we only used the criterion that $p \equiv \pm 3 \mod 8$ and $p \equiv \pm 3 \mod 10$ to build a list of possible maximal subgroups that could occur. Since we showed in the proof that no triple consisting of subgroups isomorphic to $\bbZ_{p} \rtimes \bbZ_{\frac{p-1}{2}}, D_{p-1}, D_{p+1}, A_{4}$ can constitute a triple of maximal subgroups in general position with nontrivial radical, we know (from Theorem \ref{Dickson}) that any triple of maximal subgroups in general position with nontrivial radical must contain an $S_{4}$ or $A_{5}$. That is to say, any witness to failure must lie in some $S_{4}$ or $A_{5}$. Hence, a witness to failure can have order only $2,3,4$, or $5$.

To rule out elements of order 5, consider any triple of maximal subgroups in general position: $(A_{5}, M_{1}, M_{2})$. Suppose this triple corresponds to an irredundant generating sequence and does not intersect trivially, but contains some element $x$ such that $\vert x \vert = 5$. Then $\langle x \rangle \cong \bbZ_{5} \leq A_{5}$. The only chain of subgroups of $A_{5}$ containing $\bbZ_{5}$ is as follows: 
\begin{equation}
\label{no5}
\bbZ_{5} \trianglelefteq D_{10} \leq A_{5}
\end{equation}
Thus, for the triple to be in general position, it must be the case that 
\begin{equation}
\label{no5-2}
A_{5} \cap M_{1} \cap M_{2} \cong \bbZ_{5}
\end{equation} 
and 
\begin{equation}
\label{no5-3}
A_{5} \cap M_{1} \cong D_{10}, \quad A_{5} \cap M_{2} \cong D_{10}
\end{equation}
But $N_{A_{5}}(\bbZ_{5}) \cong D_{10}$, whence Lemma \ref{norms} implies that $A_{5} \cap M_{1}  = A_{5} \cap M_{2} \cong D_{10} $ and hence, 
\begin{equation}
\label{no5-4}
A_{5} \cap M_{1} \cap M_{2} \cong D_{10} \ncong \bbZ_{5},
\end{equation}
which is a contradiction.

We also rule out an element's having order 4 by the same argument by realizing that the only chain of subgroups of $S_{4}$ of length 3 ending with a subgroup that contains an element of order 4 is $$\bbZ_{4} \trianglelefteq D_{8} \leq S_{4}$$
\end{proof}

\begin{cor}
The groups $\PSL(2,p)$ with $p \in \{ 7, 11, 19, 31\}$ satisfy $\RP$.
\end{cor}
\begin{proof}
By Theorem \ref{Jambor}, $m(\PSL(2,p)) = 4$ for these primes, it suffices by the Proposition \ref{iffCondition} that any sequence of maximal subgroup in general position of length 4 would have a trivial radical. Similar considerations as of Corollary \ref{orders} and of Lemma \ref{ChainOfSubs} leads to the following observation. Any sequence of maximal subgroup of length 3 in general position with non-trivial radical must  have $\bbZ_{2}$ or $\bbZ_{3}$ as its radical. Hence any sequence of maximal subgroups of length 4 must have trivial radical. 
\end{proof} 
\vspace{2mm}
We will actually be able to refine Corollary \ref{orders} using techniques discussed in the following proofs. Even though we will use explicit knowledge about the subgroup lattice of the groups of the type $\PSL(2,p)$, a similar conclusion can be drawn using the character table. The main tool in this case would be Mackey's theorem and similar work can be found in \cite{KomiTsar}. 

It remains to show that $\RP$ fails for $\PSL(2,p)$ in the remaining cases. To begin this endeavor, we quote a lemma from King \cite{King}:

\begin{lem}
\label{lemma:king}
\begin{enumerate}
\item There are $\frac{p(p^2-1)}{12}$ subgroups of $PSL(2,p)$ isomorphic to $S_3$.
\item If $p\equiv \pm 1 \mod 10$, there are $\frac{p(p^2-1)}{20}$ subgroups of $PSL(2,p)$ isomorphic to $D_{10}$.
\item If $p\equiv \pm 1 \mod 8$, then there are $\frac{p(p^2-1)}{24}$ subgroups of $PSL(2,p)$ isomorphic to $S_4$.
\item If $p\equiv \pm 1 \mod 10$, then there are $\frac{p(p^2-1)}{60}$ subgroups of $PSL(2,p)$ isomorphic to $A_5$.
\end{enumerate}
\end{lem}

A simple consequence of Lemma \ref{lemma:king} that is used in subsequent proofs follows:

\begin{lem}
\label{counting}
For $p\equiv \pm 1 \mod 8$ (resp. $p \equiv \pm 1 \mod 10$), $\PSL(2,p)$ always contains two isomorphic copies of $S_4$,(resp. $A_5$) that intersect in a copy $S_3$ (resp. $D_{10}$).
\end{lem}
\begin{proof}
We know that for $p\equiv \pm 1 \mod 8$, $\PSL(2,p)$ has $\frac{p(p^2-1)}{24}$ distinct subgroups isomorphic to $S_4$ (item 3 of Lemma \ref{lemma:king}). Further, each $S_{4}$ has four distinct copies of $S_{3}$. Suppose all these copies of $S_{3}$ were distinct. Then there would be at least $$4 \cdot \frac{p(p^2-1)}{24} = \frac{p(p^2-1)}{6}$$ subgroups of $\PSL(2,p)$ isomorphic to $S_{3}$. This contradicts item $1$ of Lemma \ref{lemma:king}. Therefore, there exist two copies of $S_{4}$ must intersect in an $S_{3}$. The same line of arguments with corresponding parts of Lemma \ref{counting} shows that when $p=\pm 1 \mod 10$ there exist two copies of $A_{5}$ that must intersect in an $D_{10}$.
\end{proof}
\vspace{2mm}
We can now prove the rest of Theorem \ref{Theorem1}. Our strategy for the proof will be to use the maximal subgroups in Lemma \ref{counting} and their intersections to construct irredundant generating sequences of length 3 that fail to replace an element.
\begin{prop}
\label{RPfail}
If $p \equiv \pm 1 \mod 8$ or $p \equiv \pm 1 \mod 10$ and $p \notin \{ 7, 11, 19, 31\} $, then $\PSL(2,p)$ fails $\RP$.
\end{prop}

\begin{proof}
\underline{Case 1:} $p\equiv \pm 1 \mod 8$:\\
We take two subgroups $M_1$ and $M_3$ isomorphic to $S_4$ such that their intersection is isomorphic to $S_3$. Consider $w$ an element of order 2 contained in $M_1\cap M_3$. Then there is only one subgroup $A$ of $M_1$ that is isomorphic to $\bbZ_2\times \bbZ_2$ and contains $w$. There is also a unique subgroup $B\cong \bbZ_2\times \bbZ_2$ of $M_3$ containing $w$.
We now take as $M_2$ a maximal subgroup of $\PSL(2,p)$ containing the centralizer of $w$.
More precisely, we take as $M_2$ the only subgroup isomorphic to $D_{p\mp 1}$ (plus or minus sign according to $p\equiv \pm 1 \mod 8$)$\cite{King}$ that contains both $A$ and $B$, namely, $M_2$ is the centralizer of $w$.
It is clear that the maximal subgroups $M_1$, $M_2$, and $M_3$ are in general position and have nontrivial intersection.

We take $g_1$ to be the element of order 2 in $B$ distinct from $w$ that is conjugate to $w$ in $M_3$. We take $1\neq g_3\neq w$ in $A$. Let $g_2$ be an element of order 3 in $M_1\cap M_3$. Then $g_i\in\cap_{j\neq i}M_j \setminus M_i$ and $(g_1,g_2,g_3)$ is an irredundant generating sequence of $\PSL(2,p)$ since it is irredundant by construction and $\langle g_1, g_2\rangle=M_3$, which is a maximal subgroup not containing $g_3$. Finally, note that $(g_1,g_2,g_3)$ does not satisfy the replacement property as $\langle\{g_j\}_{j\neq i} \cup\{w\}\rangle\le M_i$.

\begin{tikzcd}
 &  & M_2 \arrow[lldd, no head, bend right] \arrow[rrdd, no head, bend left] &  &  \\
 & M_1\cong S_4 \arrow[rd, no head] \arrow[ld, no head] &  & M_3\cong S_4 \arrow[ld, no head] \arrow[rd, no head] &  \\
A\cong \bbZ_2\times \bbZ_2 \arrow[dd, no head] \arrow[rrdd, no head, bend right] &  & M_1\cap M_3\cong S_3 \arrow[dd, no head] \arrow[ld, no head] &  & B\cong \bbZ_2\times \bbZ_2 \arrow[dd, no head] \arrow[lldd, no head] \\
 & \bbZ_3 \cong\langle g_2 \rangle &  &  &  \\
\bbZ_2\cong\langle g_3 \rangle &  & \bbZ_2\cong\langle w \rangle &  & \bbZ_2\cong\langle g_1 \rangle
\end{tikzcd}

\underline{Case 2:} $p\equiv \pm 1 \mod 10$:\\
We consider two subgroups $M_1$ and $M_3$ isomorphic to $A_5$ such that $M_1\cap M_3\cong D_{10}$. 
Let $w$ be an element of order 2 contained in $M_1\cap M_3$. Then there is only one subgroup $A$ of $M_1$ isomorphic to $\bbZ_2\times \bbZ_2$ and contains $w$. There is also a unique subgroup $B\cong \bbZ_2\times \bbZ_2$ of $M_3$ that contains $w$. We now take as $M_2$ a maximal subgroup of $\PSL(2,p)$ containing the centralizer of $w$ in $\PSL(2,p)$.
More precisely, we take as $M_2$ the only subgroup isomorphic to $D_{p\mp 1}$ (here the order of this dihedral group depends on $p\equiv \pm 1 \mod 4$)$\cite{King}$ which is the centralizer of $w$ and contains both $A$ and $B$. 
It is clear that the maximal subgroups $M_1$, $M_2$, and $M_3$ are in general position and have nontrivial intersection.

We take $g_1\neq w$ an element of order 2 in $B$. We take $1\neq g_3\neq w$ in $A$. Let $g_2$ be an element of order 5 in $M_1\cap M_3$. Then $g_i\in\cap_{j\neq i}M_j \setminus M_i$ and $(g_1,g_2,g_3)$ is an irredundant generating sequence of $\PSL(2,p)$ since it is irredundant by construction and $\langle g_1, g_2\rangle=M_3$, which is a maximal subgroup not containing $g_3$. Similarly, $(g_1,g_2,g_3)$ does not satisfy the replacement property.

\begin{tikzcd}
 &  & M_2 \arrow[lldd, no head, bend right] \arrow[rrdd, no head, bend left] &  &  \\
 & M_1\cong A_5 \arrow[rd, no head] \arrow[ld, no head] &  & M_3\cong A_5 \arrow[ld, no head] \arrow[rd, no head] &  \\
A\cong \bbZ_2\times \bbZ_2 \arrow[dd, no head] \arrow[rrdd, no head, bend right] &  & M_1\cap M_3\cong D_{10} \arrow[dd, no head] \arrow[ld, no head] &  & B\cong \bbZ_2\times \bbZ_2 \arrow[dd, no head] \arrow[lldd, no head] \\
 & \bbZ_5 \cong\langle g_2 \rangle &  &  &  \\
\bbZ_2\cong\langle g_3 \rangle &  & \bbZ_2\cong\langle w \rangle &  & \bbZ_2\cong\langle g_1 \rangle
\end{tikzcd}
\end{proof}

From the last proof it follows that:

\begin{cor}
\label{cor:ord2}
If the replacement property fails for $\PSL(2,p)$, then there is a witness to failure of order 2.
\end{cor}

\begin{prop}
\label{prop:ord3}
$\PSL(2,p)$ has a witness to failure of order 3 if and only if $p\equiv \pm 1 \mod 10$.
\end{prop}

\begin{proof}
If $p\equiv \pm 1 \mod 10$, then a similar argument to the one used in the previous proof shows that there exist three maximal subgroups in general position such that $M_1\cong D_{p\mp 1}$ (here the order of this dihedral group depends on $p\equiv \pm 1 \mod 3$), $M_2\cong A_5$, $M_3\cong A_5$, $M_2\cap M_3\cong A_4$, $M_1\cap M_3\cong S_3$,  $M_1\cap M_2\cong S_3$ and $\cap_{i=1}^{3} M_i \cong \bbZ_3$. This allows to construct an irredundant generating sequence $(g_1,g_2,g_3)$ with a witness to failure $w$ of order 3 such that $g_i\in\cap_{j\neq i}M_j \setminus M_i$, $\langle g_1 \rangle\neq\langle w \rangle$ and $\langle g_1, g_3\rangle=M_2$.\\

\begin{tikzcd}
 &  & M_1\cong D_{p\mp1} \arrow[lldd, no head, bend right] \arrow[rrdd, no head, bend left] &  &  \\
 & M_2\cong A_5 \arrow[rd, no head] \arrow[ld, no head] &  & M_3\cong A_5 \arrow[ld, no head] \arrow[rd, no head] &  \\
M_1\cap M_2\cong S_3 \arrow[dd, no head] \arrow[rrdd, no head, bend right] &  & M_2\cap M_3\cong A_4 \arrow[dd, no head] \arrow[ld, no head] &  & M_1\cap M_3\cong S_3 \arrow[dd, no head] \arrow[lldd, no head] \\
 & \bbZ_3 \cong\langle g_1 \rangle &  &  &  \\
\bbZ_2\cong\langle g_3 \rangle &  & \bbZ_3\cong\langle w \rangle &  & \bbZ_2\cong\langle g_2 \rangle
\end{tikzcd}

\hspace{0.5cm}

Conversely, if there is an irredundant generating sequence $(g_1,g_2,g_3)$ with a witness to failure $w$ of order 3 and $p\not\equiv \pm 1 \mod 10$, then one of the maximal subgroups of the corresponding sequence of maximal subgroups in general position must be isomorphic to $S_4$ (without loss of generality, $M_1$). This leads to a contradiction as both $M_1\cap M_2$ and $M_1\cap M_3$ would have to be equal to the only copy of $D_3$ containing $C_3\cong\langle w \rangle$.
\end{proof}

\bigbreak

{\bf Acknowledgement.} The authors are thankful to R. K. Dennis for patiently and carefully teaching us the requisite information needed for this paper and for guiding our inquiries in fruitful directions. Further thanks goes to the math department at Cornell University for hosting the SPUR/REU program and of course to our group mates without whom progress would have been far slower and less enjoyable. The first author is grateful to Mar\'ia Cristina Masaveu Peterson Foundation for their funding. The second author would also like to acknowledge and thank the Science Scholars Program at Temple University for summer funding. The third author would like to thank The Crossing Paths for their traveling grant. 

\bibliography{GSOFGbib}

\begin{thebibliography}{10}

\bibitem{CamCar}
Peter~J. Cameron and Philippe Cara.
\newblock Independent generating sets and geometries for symmetric groups.
\newblock {\em J. Algebra}, 258:641--650, 2002.

\bibitem{ColThesis}
Dan Collins.
\newblock {\em Generating Sequences of Finite Groups Senior Thesis}, 2010.

\bibitem{Dickson}
Leonard~Eugene Dickson.
\newblock {\em Linear groups: {W}ith an exposition of the {G}alois field
  theory}.
\newblock with an introduction by W. Magnus. Dover Publications, Inc., New
  York, 1958.

\bibitem{GAP}
The GAP~Group.
\newblock {\em {GAP -- Groups, Algorithms, and Programming, Version 4.9.2}},
  2018.

\bibitem{Jambor}
Sebastian Jambor.
\newblock The minimal generating sets of {${\rm PSL}(2,p)$} of size four.
\newblock {\em LMS J. Comput. Math.}, 16:419--423, 2013.

\bibitem{King}
Oliver~H. King.
\newblock The subgroup structure of finite classical groups in terms of
  geometric configurations.
\newblock In {\em Surveys in combinatorics 2005}, volume 327 of {\em London
  Math. Soc. Lecture Note Ser.}, pages 29--56. Cambridge Univ. Press,
  Cambridge, 2005.

\bibitem{KomiTsar}
Edward~A. Komissartschik and Sergey~V. Tsaranov.
\newblock Intersections of maximal subgroups in simple groups of order less
  than $10^6$.
\newblock {\em Communications in Algebra}, 14(9):1623--1678, 1986.

\bibitem{Lam}
Hy~P.~G {Lam}.
\newblock {The Replacement Property of PSL$(2,p)$ and PSL$(2,p^2)$}.
\newblock {\em arXiv e-prints}, page arXiv:1709.08745, September 2017.

\bibitem{Lucchini}
A.~{Lucchini}.
\newblock {Finite soluble groups satisfying the replacement property}.
\newblock {\em arXiv e-prints}, page arXiv:1710.00582, October 2017.

\bibitem{Nachman}
Benjamin Nachman.
\newblock Generating sequences of {$\text{PSL}(2,p)$}.
\newblock {\em J. Group Theory}, 17(6):925--945, 2014.

\bibitem{Tarski}
Alfred Tarski.
\newblock An interpolation theorem for irredundant bases of closure structures.
\newblock {\em Discrete Math.}, 12:185--192, 1975.

\end{thebibliography}
\bibliographystyle{plain}

\vspace{0.5in}

\noindent\textsc{Department of Mathematics, University of Oviedo\\
Department of Mathematics, Cornell University (2018 SPUR/REU)\\
\emph{E-mail address: }{\bf davidcuetonoval@gmail.com}}
\\
\\
\noindent\textsc{Department of Mathematics, Temple University\\
Department of Mathematics, Cornell University (2018 SPUR/REU)\\
\emph{E-mail address: }{\bf aidanlorenz@gmail.com}}
\\
\\
\noindent\textsc{Department of Mathematics, Bilkent University\\
Department of Mathematics, Cornell University (2018 SPUR/REU)\\}
\includegraphics[scale=0.13]{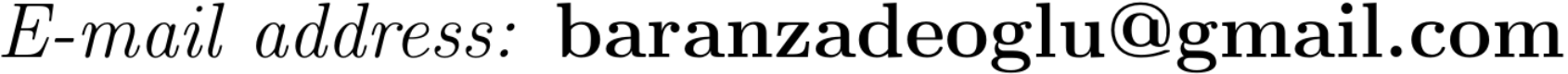}   
\end{document}